\documentclass[10pt]{amsart}
\usepackage{lipsum}
\makeatletter
\g@addto@macro{\endabstract}{\@setabstract}
\newcommand{\authorfootnotes}{\renewcommand\thefootnote{\@fnsymbol\c@footnote}}%
\makeatother
\usepackage{amssymb}
\usepackage[numbers,sort&compress]{natbib}
\renewcommand{\leq}{\leqslant}

 \makeatletter \numberwithin{equation}{section}
\numberwithin{figure}{section} 
\theoremstyle{plain}
\newtheorem{thm}{Theorem}[section]

\newtheorem{lem}[thm]{Lemma} 
\theoremstyle{Definition}
\newtheorem{rem}[thm]{Remark}

\begin{document}

\begin{center}
  \LARGE
 On the approximation properties of bivariate $(p, q)-$Bernstein operators (Revised)\par \bigskip

  \par
\bigskip

\normalsize
\authorfootnotes
Ali Karaisa\footnote{Corresponding author: Tel:+90 332 323 8220;
fax:+90 332 323 8245 E-mail:akaraisa@hotmail.com}

\par \bigskip

Department of Mathematics--Computer Sciences, Faculty of Sciences,
Necmettin Erbakan University, 42090 Meram, Konya, Turkey

\par
  \par \bigskip

\end{center}

\begin{abstract}
In the present study, we have given a corrigendum to our paper on
the approximation properties of bivariate $(p, q)-$Bernstein
operators. Recently, we \cite{kar} have defined the bivariate $(p,
q)-$Bernstein operators. Later, we have  aware of  Acar  et. al
\cite{acar} already have given some moments. In this case, we have
revised \cite[Lemma 2.3]{kar}.

\textbf{Keywords:bivariate $(p, q)-$Bernstein operator, Voronovskaja
type theorem, $(p,q)-$integer}

\textbf{MSC:}41A25, 41A36
\end{abstract}

\section{Introduction}

Approximation theory has been used in the theory of approximation of
continuous functions by means of sequences of positive linear
operators and still remains as a very active area of research. Since
Korovkin's famous theorem in 1950, the study of the linear methods
of approximation given by sequences of positive and linear operators
has became a firmly entrenched part of approximation theory.

During the last two decades, the applications of $q-$calculus have
emerged as a new area in the field of approximation theory. The
first $q-$analogue of the well-known Bernstein polynomials was
introduced by Lupa\c{s} \cite{10} by applying the idea of
$q-$integers. Since approximation studied by $q-$Bernstein
polynomials is better than classical one under convenient choice of
$q$, many authors introduced $q$-generalization of various operators
and investigated several approximation properties we refer the
readers to \cite {1,2,3}.

Recently, Mursaleen et al. used $(p, q)$-calculus in approximation
theory and defined  $(p, q)-$analogue of Bernstein operators
\cite{11}. They estimated uniform convergence of the operators and
rate of convergence, obtained Voronovskaya theorem as well. Also,
$(p, q)-$analogue of Bernstein-Stancu operators and
Bleimann-Butzer-Hahn operators were introduced in \cite{12} and
\cite{13}, respectively.  For some recent works devoted to $(p,
q)$-operators, we can refer the readers to \cite {p0,p1,p2,p3,p4}.

In the present study, we define the bivariate Bernstein  operators
based on $(p, q)-$integer. We examine  approximation properties of
our new operator by the help of Korovkin-type theorem . Further, we
present the local approximation properties and establish the rates
of  convergence by means of the modulus of continuity and the
Lipschitz type maximal function. Also, we present a Voronovskaja
type asymptotic formula for this operators.

Let us recall some definitions and notations regarding the concept
of $(p,q)-$calculus.

The $(p,q)-$integer of the number $n$ is defined by
\begin{equation*}
\left[n\right]_{p,q}:=\frac{p^{n}-q^{n}}{p-q},\,\,\,n=1,2,3\ldots,\,\,\,\,0<q<p\leq1.
\end{equation*}
The $(p,q)-$factorial $\left[ n\right] _{p,q}!$ and the
$(p,q)-$binomial coefficients are  defined as :
\begin{equation*}
\left[ n\right] _{p,q}!:=\left \{
\begin{tabular}{ll}
$\left[ n\right] _{p,q}\left[ n-1\right] _{p,q}\cdots \left[
1\right] _{p,q},$ & $
n\in \mathbb{N}$ \\
$1,$ & $n=0$
\end{tabular}
\  \right. .
\end{equation*}
and
\begin{equation*}
\left[
\begin{array}{c}
n \\
k
\end{array}
\right] _{p,q}=\frac{\left[ n\right] _{p,q}!}{\left[ k\right] _{p,q}!\left[ n-k%
\right] _{p,q}!},0\leq k\leq n.
\end{equation*}

Further, the $(p, q)-$binomial expansions are given as
\begin{equation*}
(ax+by)_{p,q}^{n}=\sum^{n}_{k=0}p^{\binom {n-k} {2}}q^{\binom {k}
{2}}a^{n-k}b^{k}x^{n-k}y^{k}.
\end{equation*}
and
\begin{equation*}
(x-y)_{p,q}^{n}=(x-y)(px-qy)(p^{2}x-q^{2}y)\cdots(p^{n-1}x-q^{n-1}y).
\end{equation*}

Further information related to $(p,q)-$calculus can be found in
\cite{14,15}.

\section{Construction of the operators}
Recently,  Mursaleen et al applied $(p, q)-$calculus in
approximation theory and introduced revised $(p, q)-$analogue of
Bernstein operators as follows;

\begin{equation}
B_{n,p,q}\left(
f;x\right)=\frac{1}{p^{n(n-1)/2}}\sum_{k=0}^{n}\left[
\begin{array}{c}
n \\
k
\end{array}
\right]
_{p,q}p^{k(k-1)/2}x^{k}\prod_{s=0}^{n-k-1}\left(p^{s}-q^{s}\frac{x}{b_{n}}\right)f\left(
\frac{\left[ k\right] _{p,q}}{\left[ n\right] _{p,q}
p^{k-n}}\right), \label{A}
\end{equation}
On the other hand, another active research area in approximation
theory is to approximate the bivariate functions. For example,
Barbosu, \cite{bar} defined and studied  bivariate Bernstein
operator. B\"{u}y\"{u}kyaz{\i}c{\i}\cite{by}  introduced
$q-$Bernstein Chlodowsky operator.

Now, we define bivariate Bernstein operator based on
$(p,q)-$integers. Let $I = [0, 1]\times[0, 1]$, $f :I\longrightarrow
R$ and $0 < q_{1}, q_{2} < p_{1}, p_{2} \leq1$. We define the
bivariate extension of the $(p, q)-$Bernstein operator operators as
follows:

\begin{equation}
B^{(p_{1},q_{1}),(p_{1},q_{1})}_{n,m}\left(
f;x,y\right)=\sum_{k=0}^{n}\sum_{j=0}^{m}R_{n,k}(p_{1},q_{1};x)R_{m,j}(p_{2},q_{2};y)
f\left( \frac{\left[ k\right] _{p_{1},q_{1}}}{\left[ n\right]
_{p_{1},q_{1}} p_{1}^{k-n}},\frac{\left[ j\right]
_{p_{2},q_{2}}}{\left[ m\right] _{p_{2},q_{2}} p_{2}^{j-m}}\right),
\label{A}
\end{equation}
where
\begin{equation*}\label{1}
R_{n,k}(p_{1},q_{1};x)=p_{1}^{\frac{k(k-1)-n(n-1)}{2}} \left[
\begin{array}{c}
n \\
k
\end{array}
\right]
_{p_{1},q_{1}}x^{k}\prod_{s=0}^{n-k-1}\left(p_{1}^{s}-q_{1}^{s}x\right)
\end{equation*}

\begin{lem}\cite[Lemma 1]{5}\label{lemma1}

\begin{eqnarray*}
B_{n,p,q}\left( e^{0};x\right)& =&1,\label{3}\\
B_{n,p,q}\left( e_{1};x\right)
&=&x,\label{4}\\
B_{n,p,q}\left( e_{2};x\right)&
=&\frac{p^{n-1}}{[n]_{p,q}}x+\frac{q[n-1]_{p,q}}{[n]_{p,q}}x^{2}
\label{5}.\\
\end{eqnarray*}
\end{lem}
Also, $(p, q)-$Bernstein operator satisfy following equations:
\begin{lem}\label{lemma2}
\begin{eqnarray*}
B_{n,p,q}\left( e_{3};x\right)&
=&\frac{p^{2n-2}}{[n]^{2}_{p,q}}x+\frac{p^{n-1}(2p+q)q[n-1]_{p,q}}{[n]^{2}_{p,q}}x^{2}
+\frac{q^{3}[n-1]_{p,q}[n-2]_{p,q}}{[ n]^{2} _{p,q}}x^{3} \label{6},\\
B_{n,p,q}\left( e_{4};x\right)& =&\frac{p^{3n-3}}{[n]^{3}_{p,q}}+\frac{q(3p^{2}+3qp+q^{3})[n-1]_{p,q}p^{2n-4}}{[n]^{3}_{p,q}}x^{2}\\
&&+\frac{q^{3}(3p^{2}+2pq+q^{2})[n-1]_{p,q}[n-2]_{p,q}p^{n-3}}{[n]^{3}_{p,q}}x^{3}+\frac{q^{6}[n-1]_{p,q}[n-2]_{p,q}[n-3]_{p,q}x^{4}}{[n]^{3}_{p,q}}
\label{6},
\end{eqnarray*} where
 $e_i(x)=x^i$, $i=0, 1, 2,3,4$.
\end{lem}

\begin{proof}
Let us we compute $e_{3}$
\begin{eqnarray*}
B_{n,p,q}\left(
e_{3};x\right)&=&\frac{1}{p^{n(n-7)/2}}\sum_{k=0}^{n}\left[
\begin{array}{c}
n \\
k
\end{array}\right]_{p,q}p^{k(k-7)/2}x^{k}\prod_{s=0}^{n-k-1}\left(p^{s}-q^{s}x\right)
\frac{[ k]^{3} _{p,q}}{\left[ n\right]^{3} _{p,q} }\\
&=&\frac{1}{p^{n(n-7)/2}\left[ n\right]^{2}
_{p,q}}\sum_{k=0}^{n-1}\left[
\begin{array}{c}
n-1 \\
k
\end{array}\right]_{p,q}p^{(k+1)(k-6)/2} [k+1]^{2}_{p,q}
x^{k+1}\prod_{s=0}^{n-k-2}\left(p^{s}-q^{s}x\right).
\end{eqnarray*}
By $[k+1]^{2}_{p,q}=p^{2k}+2qp^{k}[k]_{p,q}+q^{2}[k]^{2}_{p,q}$, we
have
\begin{eqnarray*}
&=&\frac{1}{p^{n(n-7)/2}[ n]^{2} _{p,q} }\sum_{k=0}^{n-1}\left[
\begin{array}{c}
n-1 \\
k
\end{array}\right]_{p,q}
p^{(k^{2}-k-6)/2}x^{k+1}\prod_{s=0}^{n-k-2}\left(p^{s}-q^{s}x\right)\\
&&+2q\frac{1}{p^{n(n-7)/2}[ n]^{2} _{p,q} }\sum_{k=0}^{n-1}\left[
\begin{array}{c}
n-1 \\
k
\end{array}\right]_{p,q}[k]_{p,q}
p^{(k^{2}-3k-6)/2}x^{k+1}\prod_{s=0}^{n-k-2}\left(p^{s}-q^{s}x\right)\\
&&+\frac{q^{2}}{p^{n(n-7)/2}[ n]^{2} _{p,q} }\sum_{k=0}^{n-1}\left[
\begin{array}{c}
n-1 \\
k
\end{array}\right]_{p,q}[k]^{2}_{p,q}
p^{(k+1)(k-6)/2}x^{k+1}\prod_{s=0}^{n-k-2}\left(p^{s}-q^{s}x\right)\\
&=&\frac{x}{[n]^{2}_{p,q}}p^{2n-2}+\frac{2q[n-1]_{p,q}x^{2}}{[n]^{2}_{p,q}}p^{n-1}
+\frac{p^{n-2}q^{2}[n-1]_{p,q}x^{2}}{[ n]^{2} _{p,q}}+\frac{q^{3}[n-1]_{p,q}[n-2]_{p,q}x^{3}}{[ n]^{2} _{p,q}}\\
&=&\frac{x}{[n]^{2}_{p,q}}p^{2n-2}+\frac{(2p+q)q[n-1]_{p,q}x^{2}}{[n]^{2}_{p,q}}p^{n-2}
+\frac{q^{3}[n-1]_{p,q}[n-2]_{p,q}x^{3}}{[ n]^{2} _{p,q}}.
\end{eqnarray*}
Finally,

\begin{eqnarray*}
B_{n,p,q}\left(
e_{4};x\right)&=&\frac{1}{p^{n(n-9)/2}}\sum_{k=0}^{n}\left[
\begin{array}{c}
n \\
k
\end{array}\right]_{p,q}x^{k}\prod_{s=0}^{n-k-1}\left(p^{s}-q^{s}x\right)
p^{k(k-1)/2}p^{-4k}\frac{[ k]^{4} _{p,q}}{\left[ n\right]^{4} _{p,q} }\\
&=&\frac{b^{4}_{n}}{p^{n(n-9)/2}[ n]^{3} _{p,q}
}\sum_{k=0}^{n-1}\left[
\begin{array}{c}
n-1 \\
k
\end{array}\right]_{p,q} p^{(k+1)(k-8)/2}[k+1]^{3}_{p,q}
x^{k+1}\prod_{s=0}^{n-k-2}\left(p^{s}-q^{s}x\right).
\end{eqnarray*}
Using the fact
$[k+1]^{3}_{p,q}=p^{3k}+3p^{2k}q[k]_{p,q}+3p^{k}q^{2}[k]^{2}_{p,q}+q^{3}[k]^{3}_{p,q}$,
we obtain

\begin{eqnarray*}
B_{n,p,q}\left( e_{4};x\right)&=&\frac{p^{3n-3}}{p^{(n-1)(n-2)/2}[
n]^{3} _{p,q}}\sum_{k=0}^{n-1}\left[
\begin{array}{c}
n-1 \\
k
\end{array}\right]_{p,q}p^{k(k-1)/2}x^{k+1}\prod_{s=0}^{n-k-2}\left(p^{s}-q^{s}x\right)\\
&&+\frac{3qb^{4}_{n}}{p^{n(n-9)/2}[ n]^{3}
_{p,q}}\sum_{k=0}^{n-1}\left[
\begin{array}{c}
n-1 \\
k
\end{array}\right]_{p,q}[k]_{p,q}p^{(k^{2}-3k-8)/2}x^{k+1}\prod_{s=0}^{n-k-2}\left(p^{s}-q^{s}x\right)\\
&&+\frac{3q^{2}}{p^{n(n-9)/2}[ n]^{3} _{p,q}}\sum_{k=0}^{n-1}\left[
\begin{array}{c}
n-1 \\
k
\end{array}\right]_{p,q}[k]^{2}_{p,q}p^{(k^{2}-5k-8)/2}x^{k+1}\prod_{s=0}^{n-k-2}\left(p^{s}-q^{s}x\right)\\
&&+\frac{q^{3}}{p^{n(n-9)/2}[ n]^{3} _{p,q}}\sum_{k=0}^{n-1}\left[
\begin{array}{c}
n-1 \\
k
\end{array}\right]_{p,q}[k]^{3}_{p,q} p^{(k+1)(k-8)/2}x^{k+1}\prod_{s=0}^{n-k-2}\left(p^{s}-q^{s}x\right)\\
&=&\frac{x}{[n]^{3}_{p,q}}p^{3n-3}+\frac{q(3p^{2}+3qp+q^{2})[n-1]_{p,q}x^{2}}{[n]^{3}_{p,q}}p^{2n-4}\\
&&+\frac{q^{3}(3p^{2}+2pq+q^{2})[n-1]_{p,q}[n-2]_{p,q}x^{3}}{[n]^{3}_{p,q}}p^{n-3}+\frac{q^{6}[n-1]_{p,q}[n-2]_{p,q}[n-3]_{p,q}x^{4}}{[n]^{3}_{p,q}}.
\end{eqnarray*}

Acar  et. al  \cite{acar} introduced Kantotovich modifications of
$(p, q)-$Bernstein operators for bivariate functions using a new
$(p,q)-$integral and given following moments for bivariate
$(p,q)-$Bernstein operators.
\begin{lem}\cite[Lemma 1]{acar}\label{lemx}
\begin{eqnarray*}
B^{(p_{1},q_{1}),(p_{2},q_{2})}_{n,m}\left(1;x,y\right)& =&1,\label{1a}\\
B^{(p_{1},q_{1}),(p_{2},q_{2})}_{n,m}\left( s;x,y\right)&=&x,\label{2a}\\
B^{(p_{1},q_{1}),(p_{2},q_{2})}_{n,m}\left( t;x,y\right)&=&y,\label{3a}\\
B^{(p_{1},q_{1}),(p_{2},q_{2})}_{n,m}\left(st;x,y\right)&=&x y,\label{4a}\\
B^{(p_{1},q_{1}),(p_{2},q_{2})}_{n,m}\left( s^{2};x,y\right)&
=&\frac{p_{1}^{n-1}}{[n]_{p_{1},q_{1}}}x+\frac{q_{1}[n-1]_{p_{1},q_{1}}}{[n]_{p_{1},q_{1}}}x^{2}
\label{5a},\\
B^{(p_{1},q_{1}),(p_{2},q_{2})}_{n,m}\left( t^{2};x,y\right)&
=&\frac{p_{2}^{m-1}}{[m]_{p_{2},q_{2}}}y+\frac{q_{2}[m-1]_{p_{2},q_{2}}}{[n]_{p_{2},q_{2}}}y^{2}
\label{6a}.\\
\end{eqnarray*}
\end{lem}

Using Lemma \ref{lemx} and by linearity of
$B^{(p_{1},q_{1}),(p_{2},q_{2})}_{n,m}$, we have

\begin{rem}
\begin{eqnarray}
B^{(p_{1},q_{1}),(p_{2},q_{2})}_{n,m}\left((t-x)^{2};x,y\right)&
=&\frac{p_{1}^{n-1}}{[n]_{p_{1},q_{1}}}(x-x^{2}),\label{moment1}\\
B^{(p_{1},q_{1}),(p_{2},q_{2})}_{n,m}\left((s-y)^{2};x,y\right)&
=&\frac{p_{2}^{m-1}}{[m]_{p_{2},q_{2}}}(y-y^{2}).\label{moment2}
\end{eqnarray}

\end{rem}

The Korovkin-type theorem for functions of two variables was proved
by Volkov \cite{volkov}.

\begin{thm}
Let $q_{1}:=(q_{1,n})$, $p_{1}:= (p_{1,n})$, $q_{2}:=(q_{m,2})$,
$p_{2}:=(p_{m,2})$ such that
$0<q_{1,n},q_{2,m}<p_{n,1},p_{m,2}\leq1$. If
\begin{eqnarray}\label{convergent}
\lim_n p_{n,1}=1,\,\, \lim_n q_{n,1}=1,\,\,\lim_m p_{m,2}=1,\,\,
\lim_m q_{m,2}=1,\,\,\lim_n
p^{n}_{n,1}=a_{1}\,\,\textrm{and}\,\,\lim_m p^{m}_{m,1}=a_{2},
\end{eqnarray}
the sequence
$B^{(p_{1},q_{1}),(p_{2},q_{2})}_{n,m}\left(f;x,y\right)$
convergence uniformly to $f(x,y)$, on $[0,1]\times[0,1]=[0,1]^{2}$
for each $f\in C\left([0,1]^{2}\right)$, where $a_{1},a_{2}$ be reel
numbers and  $C([0,1]^{2})$ be the space of all real valued
continuous function on $[0,1]^{2}$ with the norm
\begin{eqnarray*}
\parallel f\parallel_{C([0,1]^{2})}=\sup_{(x,y)\in
[0,1]^{2}}\left|f(x,y)\right|.
\end{eqnarray*}
\end{thm}

\begin{proof}
Assume that  the equities \eqref{convergent} are holds. Then, we
have
\begin{eqnarray*}
\frac{p_{1,n}^{n-1}}{\left[n\right]_{p_{1,n},q_{1,n}}}\rightarrow
0,\, \frac{p_{2}^{m-1}}{\left[m\right]_{p_{2,m},q_{2,m}}}\rightarrow
0,\,\frac{q_{1,n}\left[n\right]_{p_{1,n},q_{1,n}}}{\left[n\right]_{p_{1,n},q_{1,n}}}\rightarrow
1\,\textrm{and}\,\,\frac{q_{2,m}\left[m-1\right]_{p_{2,m},q_{2,m}}}{\left[m\right]_{p_{2,m},q_{2,m}}}\rightarrow1.
\end{eqnarray*}

From Lemma \ref{lemx}, we obtain  $\lim_{n,m \to \infty}
B^{(p_{1},q_{1}),(p_{2},q_{2})}_{n,m}\left(e_{ij};x,y\right) =
e_{ij}(x, y)$ uniformly on $[0, 1]^{2}$, where $e_{ij}(x, y) =
x_{i}y_{j}, 0 \leq i + j \leq 2$ are the test functions. By using
Korovkin theorem for functions of two variables was  presented by
Volkov \cite{volkov}, it follows that
$\lim_{n,m\to\infty}B^{(p_{1},q_{1}),(p_{2},q_{2})}_{n,m}\left(f;x,y\right)
= f(x, y)$, uniformly on $[0, 1]^{2}$, for each $f \in
C([0,1]^{2})$.
\end{proof}

\section{Rate of Convergence}

In this section, we compute the rates of convergence of operators
$B^{(p_{1},q_{1}),(p_{2},q_{2})}_{n,m}$ to $f (x, y)$ by means of
the modulus of continuity. Proceeding further, we  provide a summary
of the notations and definitions of the modulus of continuity and
the Peetre's $K-$functional for bivariate real valued functions.

For $f \in C( [0,1]^{2})$, the complete modulus of continuity for a
bivariate case is defined as follows:

\begin{equation*}
\omega(f,\delta)=\sup\left\{|f(t,s)-f(x,y)|:\sqrt{(t-x)^{2}+(s-y)^{2}}\leq\delta\right\}.
\end{equation*}
for every $(t, s),(x, y)\in [0,1]^{2}$. Further, partial moduli of
continuity with respect to $x$ and $y$ are defined as

\begin{eqnarray*}
\omega^{1}(f,\delta)&=&\sup\left\{|f(x_{1},y)-f(x_{2},y)|:y\in[0,1]\,\,\textrm{and}\,\,|x_{1}-x_{2}|\leq\delta\right\}\\
\omega^{2}(f,\delta)&=&\sup\left\{|f(x,y_{1})-f(x,y_{2})|:x\in[0,1]\,\,\textrm{and}\,\,|y_{1}-y_{2}|\leq\delta\right\},
\end{eqnarray*}

It is obvious that they satisfy the properties of the usual modulus
of continuity \cite{ans}.

For $\delta> 0$, the Peetre-K functional \cite{peetre} is given by
\begin{eqnarray*}
K(f,\delta)=\inf_{g\in
C^{2}[0,1]^{2}}\left\{\|f-g\|_{C[0,1]^{2}}+\delta\|g\|_{C^{2}[0,1]^{2}}\right\},
\end{eqnarray*}
where $C^{2}[0,1]^{2}$ is the space of functions of $f$ such that
$f$, $\frac{\partial^{j}f}{\partial x^{j}}$ and
$\frac{\partial^{j}f}{\partial y^{j}}$ $(j=1,2)$ in $C[0,1]^{2}$.
The norm $\|.\|$ on the space  $C^{2}[0,1]^{2}$ is defined by

\begin{eqnarray*}
\|f\|_{C^{2}[0,1]^{2}}=\|f\|_{C[0,1]^{2}}+\sum_{j=1}^{2}\left(\left\|\frac{\partial^{j}f}{\partial
y^{j}}\right\|_{C[0,1]^{2}}+\left\|\frac{\partial^{j}f}{\partial
y^{j}}\right\|_{C[0,1]^{2}}\right).
\end{eqnarray*}

 Now, we give an estimate of the rate of
convergence of operators $B^{(p_{1},q_{1}),(p_{2},q_{2})}_{n,m}$.

\begin{thm}
\label{thorem3.1}Let $f\in C\left([ 0,1]^{2} \right) $. For all
 $x\in [0,1]^{2}$, we have
\begin{equation*}
\left\vert B^{(p_{1},q_{1}),(p_{2},q_{2})}_{n,m} -f\left(x, y\right)
\right\vert \leq 2\omega\left( f;\delta_{n,m} \right) ,
\label{10.2}
\end{equation*}
where
\begin{eqnarray*}
\delta_{n,m}^2&=&\frac{p_{1}^{n-1}}{[n]_{p_{1},q_{1}}}(x-x^{2})+\frac{p_{2}^{m-1}}{[m]_{p_{2},q_{2}}}(y-y^{2}).
\end{eqnarray*}
\end{thm}

\begin{proof}
By definition  the complete modulus of continuity of $f(x, y)$ and
linearity and positivity our operator, we can write
\begin{eqnarray*}
|B^{(p_{1},q_{1}),(p_{2},q_{2})}_{n,m}\left(f;x,y\right)-f(x,y)|&\leq&
B^{(p_{1},q_{1}),(p_{2},q_{2})}_{n,m}\left(|f(t,s)-f(x,y)|;x,y\right)\\
&\leq&B^{(p_{1},q_{1}),(p_{2},q_{2})}_{n,m}\left(\omega\left(f;\sqrt{(t-x)^{2}+(s-y)^{2}}\right);x,y\right)\\
&\leq&\omega(f,\delta_{n,m})\left[\frac{1}{\delta_{n,m}}B^{(p_{1},q_{1}),(p_{2},q_{2})}_{n,m}\left(\sqrt{(t-x)^{2}+(s-y)^{2}};x,y\right)\right].
\end{eqnarray*}
Using Cauchy-Scwartz inequality, from \eqref{moment1} and
\eqref{moment2}, one can write following
\begin{eqnarray*}
&&|B^{(p_{1},q_{1}),(p_{2},q_{2})}_{n,m}\left(f;x,y\right)-f(x,y)|\\
&\leq& \omega(f,\delta_{n,m})
\left[1+\frac{1}{\delta_{n,m}}\left\{B^{(p_{1},q_{1}),(p_{2},q_{2})}_{n,m}\left((t-x)^{2}+(s-y)^{2};x,y\right)\right\}^{1/2}\right]\\
&=&\omega(f,\delta_{n,m})
\left[1+\frac{1}{\delta_{n,m}}\left\{B^{(p_{1},q_{1}),(p_{2},q_{2})}_{n,m}\left((t-x)^{2};x,y\right)
+B^{(p_{1},q_{1}),(p_{2},q_{2})}_{n,m}\left((s-y)^{2};x,y\right)\right\}^{1/2}\right]\\
&=&\omega(f,\delta_{n,m})
\left[1+\frac{1}{\delta_{n,m}}\left(\frac{p_{1}^{n-1}}{[n]_{p_{1},q_{1}}}(x-x^{2})+\frac{p_{2}^{m-1}}{[m]_{p_{2},q_{2}}}(y-y^{2})\right)^{1/2}\right].
\end{eqnarray*}
Choosing  $\delta_{ n,m} =
\left(\frac{p_{1}^{n-1}}{[n]_{p_{1},q_{1}}}(x-x^{2})+\frac{p_{2}^{m-1}}{[m]_{p_{2},q_{2}}}(y-y^{2})\right)^{1/2}$,
for all $(x, y)\in[0,1]^{2}$, we get desired the result.

\end{proof}

\begin{thm}
Let $f\in C\left([ 0,1]^{2} \right) $, then the following
inequalities satisfy
\begin{equation*}
\left\vert B^{(p_{1},q_{1}),(p_{2},q_{2})}_{n,m} -f\left(x, y\right)
\right\vert \leq \omega^{1}\left( f;\delta_{n} \right)+\omega^{2
}\left( f;\delta_{m} \right),
\end{equation*}
where
\begin{eqnarray}
\delta_{n}^2&=&\frac{p_{1}^{n-1}}{[n]_{p_{1},q_{1}}}(x-x^{2}),\label{lip1}\\
\delta_{m}^2&=&\frac{p_{2}^{m-1}}{[m]_{p_{2},q_{2}}}(y-y^{2}).\label{lip2}.
\end{eqnarray}
\end{thm}

\begin{proof}
By definition partial moduli of continuity of $f(x,y)$ and appliying
Cauchy-Scwartz inequality, we have
\begin{eqnarray*}
|B^{(p_{1},q_{1}),(p_{2},q_{2})}_{n,m}\left(f;x,y\right)-f(x,y)|&\leq&
B^{(p_{1},q_{1}),(p_{2},q_{2})}_{n,m}\left(|f(t,s)-f(x,y)|;x,y\right)\\
&\leq&B^{(p_{1},q_{1}),(p_{2},q_{2})}_{n,m}\left(|f(t,s)-f(x,s)|;x,y\right)+B^{(p_{1},q_{1}),(p_{2},q_{2})}_{n,m}\left(|f(x,s)-f(x,y)|;x,y\right)\\
&\leq&B^{(p_{1},q_{1}),(p_{2},q_{2})}_{n,m}\left(|\omega^{1}(f;
|t-x|)|;x,y\right)+B^{(p_{1},q_{1}),(p_{2},q_{2})}_{n,m}\left(|\omega^{2}(f;
|s-y|)|;x,y\right)\\
&\leq&\omega^{1}(f,\delta_{n})
\left[1+\frac{1}{\delta_{n}}B^{(p_{1},q_{1}),(p_{2},q_{2})}_{n,m}\left(|t-x|;x,y\right)\right]\\
&&+\omega^{2}(f,\delta_{m})
\left[1+\frac{1}{\delta_{m}}B^{(p_{1},q_{1}),(p_{2},q_{2})}_{n,m}\left(|s-y|;x,y\right)\right]\\
&\leq&\omega^{1}(f,\delta_{n})
\left[1+\frac{1}{\delta_{n}}\left(B^{(p_{1},q_{1}),(p_{2},q_{2})}_{n,m}\left((t-x)^{2};x,y\right)\right)^{1/2}\right]\\&&+\omega^{2}(f,\delta_{m})
\left[1+\frac{1}{\delta_{m}}\left(B^{(p_{1},q_{1}),(p_{2},q_{2})}_{n,m}\left((s-y)^{2};x,y\right)\right)^{1/2}\right].
\end{eqnarray*}
Consider \eqref{moment1}, \eqref{moment2} and choosing
\begin{eqnarray*}
\delta_{n}^2&=&\frac{p_{1}^{n-1}}{[n]_{p_{1},q_{1}}}(x-x^{2}),\\
\delta_{m}^2&=&\frac{p_{2}^{m-1}}{[m]_{p_{2},q_{2}}}(y-y^{2}).
\end{eqnarray*}
we reach the result.
\end{proof}

For $\alpha_{1},\alpha_{1}\in (0,1]$ and $(s,t),(x,y)\in [0,1]^{2}$,
we define the Lipschitz class $Lip M(\alpha_{1}, \alpha_{1})$ for
the bivariate case as follows:
\begin{eqnarray*}
\left|f(s,t)-f(x,y)\right|\leq
M\left|s-x\right|^{\alpha_{1}}\left|t-y\right|^{\alpha_{2}}.
\end{eqnarray*}

\begin{thm}
Let  $f \in Lip_{M}(\alpha_{1},\alpha_{2})$. Then, for all $(x,y)\in
[0,1]^{2}$, we have

\begin{equation*}
|B^{(p_{1},q_{1}),(p_{2},q_{2})}_{n,m}\left(f;x,y\right)-f(x,y)|\leq
M\delta_{n}^{\alpha_{1}/2}\delta_{m}^{\alpha_{2}/2},
\end{equation*}
where $\delta_{n}$ and $\delta_{m}$ defined in \eqref{lip1} and
\eqref{lip2}, respectively.
\end{thm}
\begin{proof}
As $f \in Lip_{M}(\alpha_{1},\alpha_{2})$, it follows
\begin{eqnarray*}
|B^{(p_{1},q_{1}),(p_{2},q_{2})}_{n,m}\left(f;x,y\right)-f(x,y)|&\leq&
B^{(p_{1},q_{1}),(p_{2},q_{2})}_{n,m}(|f(t,s)-f(x,y)|,q_{n};x,y)\\
&\leq&M
B^{(p_{1},q_{1}),(p_{2},q_{2})}_{n,m}(|t-x|^{\alpha_{1}}|s-y|^{\alpha_{2}};x,y)\\
&=&MB^{(p_{1},q_{1}),(p_{2},q_{2})}_{n,m}(|t-x|^{\alpha_{1}}|;x)B^{(p_{1},q_{1}),(p_{2},q_{2})}_{n,m}(|s-y|^{\alpha_{2}};y).
\end{eqnarray*}
For $\widehat{p}=\frac{1}{\alpha_{1}},
\widehat{q}=\frac{\alpha_{1}}{2-\alpha_{1}}$ and
$\widehat{p}=\frac{1}{\alpha_{2}},
\widehat{q}=\frac{\alpha_{2}}{2-\alpha_{2}}$ applying the
H\"{o}lder's inequality, we get
\begin{eqnarray*}
|B^{(p_{1},q_{1}),(p_{2},q_{2})}_{n,m}\left(f;x,y\right)-f(x,y)|
&\leq&M\{B^{(p_{1},q_{1}),(p_{2},q_{2})}_{n,m}(|t-x|^{2};x)\}^{\alpha_{1}/2}\{B^{(p_{1},q_{1}),(p_{2},q_{2})}_{n,m}(1;x)\}^{\alpha_{1}/2}\\
&&\times\{B^{(p_{1},q_{1}),(p_{2},q_{2})}_{n,m}(|s-y|^{2};y)\}^{\alpha_{12}/2}\{B^{(p_{1},q_{1}),(p_{2},q_{2})}_{n,m}(1;y)\}^{\alpha_{2}/2}\\
&=&M\delta_{n}^{\alpha_{1}/2}\delta_{m}^{\alpha_{2}/2}.
\end{eqnarray*}
Hence, we get desired the result. \end{proof}

\begin{thm}
Let $f\in C^{1}([0,1]^{2})$ and
$0<q_{1,n},q_{2,m}<p_{n,1},p_{m,2}\leq1$. Then, we have
\begin{eqnarray*}
|B^{(p_{1},q_{1}),(p_{2},q_{2})}_{n,m}\left(f;x,y\right)-f(x,y)|&\leq&\parallel
f^{'}_{x}\parallel_{C([0,1]^{2})}\delta_{n}+\parallel
f^{'}_{y}\parallel_{C([0,1]^{2})}\delta_{m}.
\end{eqnarray*}

\end{thm}
\begin{proof}
For $(t,s)\in [0,1]^{2}$, we obtain

\begin{eqnarray*}
f(t)-f(s)=\int_{x}^{t}f^{'}_{u}(u,s)du+\int_{y}^{s}f^{'}_{v}(x,v)du
\end{eqnarray*}
Applying the our operator on both sides above equation,  we deduce
\begin{eqnarray*}
|B^{(p_{1},q_{1}),(p_{2},q_{2})}_{n,m}\left(f;x,y\right)-f(x,y)|&\leq&
B^{(p_{1},q_{1}),(p_{2},q_{2})}_{n,m}\left(\left|\int_{x}^{t}f^{'}_{u}(u,s)du\right|;x,y\right)\\
&&+B^{(p_{1},q_{1}),(p_{2},q_{2})}_{n,m}\left(\left|\int_{y}^{s}f^{'}_{v}(x,v)du\right|;x,y\right).
\end{eqnarray*}
As
\begin{eqnarray*}
\left|\int_{x}^{t}f^{'}_{u}(u,s)du\right| \leq\parallel
f^{'}_{x}\parallel_{C([0,1]^{2})}|t-x|\,\,\textrm{and}\,\,\left|\int_{y}^{s}f^{'}_{v}(x,v)du\right|\leq
f^{'}_{y}\parallel_{C([0,1]^{2})}|s-y|,
\end{eqnarray*}
we have
\begin{eqnarray*}
|B^{(p_{1},q_{1}),(p_{2},q_{2})}_{n,m}\left(f;x,y\right)-f(x,y)|&\leq&f^{'}_{x}\parallel_{C([0,1]^{2})}B^{(p_{1},q_{1}),(p_{2},q_{2})}_{n,m}\left(|t-x|;x,y\right)\\
&&+f^{'}_{y}\parallel_{C([0,1]^{2})}B^{(p_{1},q_{1}),(p_{2},q_{2})}_{n,m}\left(|s-y|;x,y\right).
\end{eqnarray*}
Using the Cauchy-Schwarz inequality, we can write following
\begin{eqnarray*}
|B^{(p_{1},q_{1}),(p_{2},q_{2})}_{n,m}\left(f;x,y\right)-f(x,y)|&\leq&\parallel f^{'}_{x}\parallel_{C([0,1]^{2})}\{B^{(p_{1},q_{1}),(p_{2},q_{2})}_{n,m}\left((t-x)^{2};x,y\right)\}^{1/2}\{B^{(p_{1},q_{1}),(p_{2},q_{2})}_{n,m}\left(1;x,y\right)\}^{1/2}\\
&&+\parallel
f^{'}_{y}\parallel_{C([0,1]^{2})}\{B^{(p_{1},q_{1}),(p_{2},q_{2})}_{n,m}\left((s-y)^{2};x,y\right)\}^{1/2}\{B^{(p_{1},q_{1}),(p_{2},q_{2})}_{n,m}\left(1;x,y\right)\}^{1/2}.
\end{eqnarray*}
Form \eqref{moment1} and \eqref{moment2}, we get desired the result.
\end{proof}

\begin{thm}
Let $f\in C\left([ 0,1]^{2} \right)$, then we have
\begin{eqnarray*}
\left\|B^{(p_{1},q_{1}),(p_{2},q_{2})}_{n,m}\left(f;x,y\right)-f(x,y)\right\|_{C\left([
0,1]^{2} \right)}\leq2M\left(f;\delta_{n,m}(x,y)/2 \right),
\end{eqnarray*}
where
\begin{eqnarray*}
\delta_{n,m}(x,y)=\frac{1}{2}\max\left(\frac{p_{1}^{n-1}(x-x^{2})}{[n]_{p_{1},q_{1}}},\frac{p_{2}^{m-1}(y-y^{2}}{[m]_{p_{2},q_{2}}}\right).
\end{eqnarray*}
\end{thm}

\begin{proof}
Let $g \in C^{2}([0,1]^{2})$. By the Taylor's formula, we get
\begin{eqnarray*}
g(s_{1},s_{2})-g(x,y)&=&g(s_{1},y)-g(x,y)+g(s_{1},s_{2})-g(s_{1},y)\\
&=&\frac{\partial g(x,y)}{\partial
x}(s_{1}-x)+\int_{x}^{s_{1}}(s_{1}-u)\frac{\partial^{2}
g(u,y)}{\partial u^{2}}du\\
&&+\frac{\partial g(x,y)}{\partial
x}(s_{2}-y)+\int_{y}^{s_{2}}(s_{2}-v)\frac{\partial^{2}
g(x,v)}{\partial v^{2}}dv\\
&=&\frac{\partial g(x,y)}{\partial
x}(s_{1}-x)+\int_{0}^{s_{1}-x}(s_{1}-x-u)\frac{\partial^{2}
g(u,y)}{\partial u^{2}}du\\
&&+\frac{\partial g(x,y)}{\partial
x}(s_{2}-y)+\int_{0}^{s_{2}-y}(s_{2}-y-v)\frac{\partial^{2}
g(x,v)}{\partial v^{2}}dv
\end{eqnarray*}
Applying $B^{(p_{1},q_{1}),(p_{2},q_{2})}_{n,m}$ to the both sides
of the above equation, we obtain
\begin{eqnarray*}
\left|B^{(p_{1},q_{1}),(p_{2},q_{2})}_{n,m}\left(g;x,y\right)-g(x,y)\right|&\leq&\left|\frac{\partial
g(x,y)}{\partial x}\right|\left|B^{(p_{1},q_{1}),(p_{2},q_{2})}_{n,m}\left((s_{1}-x);x,y\right)\right|\\
&&+\left|B^{(p_{1},q_{1}),(p_{2},q_{2})}_{n,m}\left(\int_{0}^{s_{1}-x}(s_{1}-x-u)\frac{\partial^{2}
g(u,y)}{\partial u^{2}}du;x,y\right)\right|\\
&&+\left|\frac{\partial
g(x,y)}{\partial y}\right|\left|B^{(p_{1},q_{1}),(p_{2},q_{2})}_{n,m}\left((s_{2}-y);x,y\right)\right|\\
&&+\left|B^{(p_{1},q_{1}),(p_{2},q_{2})}_{n,m}\left(\int_{0}^{s_{2}-y}(s_{2}-y-v)\frac{\partial^{2}
g(v,x)}{\partial v^{2}}dv;x,y\right)\right|
\end{eqnarray*}
As
$B^{(p_{1},q_{1}),(p_{2},q_{2})}_{n,m}\left((s_{1}-x);x,y\right)=0$
and
$B^{(p_{1},q_{1}),(p_{2},q_{2})}_{n,m}\left((s_{2}-y);x,y\right)=0$,
one can write following
\begin{eqnarray*}
\left\|B^{(p_{1},q_{1}),(p_{2},q_{2})}_{n,m}\left(f;x,y\right)-f(x,y)\right\|_{C\left([
0,1]^{2} \right)}&\leq& \frac{1}{2}\left\|\frac{\partial
g(x,y)}{\partial x}\right\|_{C\left([ 0,1]^{2}
\right)}\left|B^{(p_{1},q_{1}),(p_{2},q_{2})}_{n,m}\left((s_{1}-x)^{2};x,y\right)\right|\\
&&+\frac{1}{2}\left\|\frac{\partial g(x,y)}{\partial
y}\right\|_{C\left([ 0,1]^{2}
\right)}\left|B^{(p_{1},q_{1}),(p_{2},q_{2})}_{n,m}\left((s_{2}-y)^{2};x,y\right)\right|.
\end{eqnarray*}
By \eqref{moment1}, \eqref{moment2}, we deduce,

\begin{eqnarray}\label{ali1}
\left\|B^{(p_{1},q_{1}),(p_{2},q_{2})}_{n,m}\left(f;x,y\right)-f(x,y)\right\|_{C\left([
0,1]^{2}
\right)}&\leq&\frac{1}{2}\max\left(\frac{p_{1}^{n-1}(x-x^{2})}{[n]_{p_{1},q_{1}}},\frac{p_{2}^{m-1}(y-y^{2}}{[m]_{p_{2},q_{2}}}\right)\nonumber\\
&&\times\left[\left\|\frac{\partial g(x,y)}{\partial
x}\right\|_{C\left([ 0,1]^{2} \right)}+\left\|\frac{\partial
g(x,y)}{\partial x}\right\|_{C\left([ 0,1]^{2} \right)}\right]\nonumber\\
&\leq&\parallel g\parallel_{C\left([ 0,1]^{2} \right)}\delta_{n,m}.
\end{eqnarray}
By the linearity $B^{(p_{1},q_{1}),(p_{2},q_{2})}_{n,m}$, we obtain
\begin{eqnarray}\label{ali2}
\left\|B^{(p_{1},q_{1}),(p_{2},q_{2})}_{n,m}\left(f;x,y\right)-f(x,y)\right\|_{C\left([
0,1]^{2}
\right)}&\leq&\left\|B^{(p_{1},q_{1}),(p_{2},q_{2})}_{n,m}f-B^{(p_{1},q_{1}),(p_{2},q_{2})}_{n,m}g\right\|_{C\left([
0,1]^{2}
\right)}\nonumber\\&&+\left\|B^{(p_{1},q_{1}),(p_{2},q_{2})}_{n,m}g-g\right\|_{C\left([
0,1]^{2} \right)}+\left\|f-g\right\|_{C\left([ 0,1]^{2} \right)}.
\end{eqnarray}
By \eqref{ali1} and \eqref{ali2}, one can see that
\begin{eqnarray*}
\left\|B^{(p_{1},q_{1}),(p_{2},q_{2})}_{n,m}\left(f;x,y\right)-f(x,y)\right\|_{C\left([
0,1]^{2} \right)}\leq2M\left(f;\delta_{n,m}(x,y)/2 \right).
\end{eqnarray*}
This step completes the proof.
\end{proof}

First, we need the auxiliary result contained in the following
lemma.

\begin{lem}\label{lemmaVAR} Let  $0<q_{n}< p_{n}\leq1$,  be sequence such that $p_{n},q_{n}\longrightarrow 1$ and
$p^{n}_{n}\longrightarrow a$, $q^{n}_{n}\longrightarrow b$ as
$n\longrightarrow\infty$. Then, we have the following limits:
\begin{enumerate}
\item[(i)]$\lim_{n \to \infty} [n]_{q_{n}}B^{(p_{n},q_{n})}_{n}((t-x)^{2};x)=ax-ax^{2}$
\item[(ii)]$\lim_{n \to \infty} [n]_{q_{n}}^{2} B^{(p_{n},q_{n})}_{n}((t-x)^{4};x)=3ax^{4}-6ax^{3}+3ax^{2}$.
\end{enumerate}
\end{lem}
\begin{proof}
From \eqref{moment1}, we get
\begin{eqnarray*}
[n]_{p_{n},q_{n}}
B^{(p_{n},q_{n})}_{n}((t-x)^{2};x)=-p_{n}^{n-1}x^{2}+xp_{n}^{n-1}.
\end{eqnarray*}
Let us take the limit of both sides of the above equality as
$n\longrightarrow \infty$, then we have
\begin{eqnarray*}
\lim_{n \to \infty}[n]_{p_{n},q_{n}}
\left\{B^{(p_{n},q_{n})}_{n}((t-x)^{2},x)\right\}&=&\lim_{n \to
\infty}
\left\{-p_{n}^{n-1}x^{2}+xp_{n}^{n-1}\right\}\\
&=&a(x-x^{2}).
\end{eqnarray*}

(ii) Using Lemma \ref{lemma1},  Lemma \ref{lemma2}  and by the
linearity of the operators $B^{(p_{n},q_{n})}_{n}(f;x)$, we obtain
\begin{eqnarray*} \
B^{(p_{n},q_{n})}_{n}(t-x)^{4};x)=A_{1,n}x^{4}+A_{2,n}x^{3}+A_{3,n}x^{2}+A_{4,n}x
\end{eqnarray*}
where
\begin{eqnarray*}
A_{1,n}&=&\frac{p_{n}^{n-3}[n]^{2}_{p_{n},q_{n}}(-p_{n}^{2}+2p_{n}q_{n}-q_{n}^{2})+p^{n-5}[n]_{p_{n},q_{n}}(-p_{n}^{3}+3p_{n}q_{n}^{2}+
q_{n}^{3})-p_{n}^{3n-6}(p_{n}^{2}+p_{n}^{3}+2p_{n}q_{n}^{2}+q_{n}^{3})}{[n]^{3}_{p_{n},q_{n}}}\\
A_{2,n}&=&\frac{p_{n}^{n-3}[n]^{2}_{p_{n},q_{n}}(p_{n}^{2}-2p_{n}q_{n}+q_{n}^{2})}{[n]^{3}_{p_{n},q_{n}}}\\
&&+\frac{p^{2n-5}[n]_{p_{n},q_{n}}(-q_{n}^{3}-4p_{n}q_{n}^{2}-3p^{2}_{n}q_{n}+2p_{n}^{3})-p_{n}^{3n-6}(3p_{n}^{3}+3p_{n}q_{n}^{2}+5p^{2}_{n}q_{n}+q_{n}^{3})}{[n]^{3}_{p_{n},q_{n}}}\\
A_{3,n}&=&\frac{p_{n}^{2n-4}[n]_{p_{n},q_{n}}(-p_{n}^{2}+3p_{n}q_{n}+q_{n}^{2})-p_{n}^{3n-5}(3p_{n}^{2}+q_{n}^{2}+3p_{n}q_{n})}{[n]^{3}_{q_{n}}}\\
A_{4,n}&=&\frac{p^{3n-3}}{[n]^{3}_{q_{n}}},\\
\end{eqnarray*}

It is clear that
\begin{eqnarray}\label{ali1}
\lim_{n \to \infty} [n]_{q_{n}}^{2}\{A_{4,n}x\}=0.
\end{eqnarray}

Taking the limit of both sides of $A_{1,n}$, we get
\begin{eqnarray}\label{ali2}
\lim_{n\to\infty}[n]_{q_{n}}^{2}\{A_{1,n}\}&=&\lim_{n\to\infty}\left\{-p_{n}^{n-3}[n]_{p_{n},q_{n}}(p_{n}-q_{n})^{2}
+p_{n}^{n-5}(-p_{n}^{3}+3p_{n}q_{n}^{2}+q_{n}^{3})-\frac{p_{n}^{3n-6}(p_{n}^{2}+p_{n}^{3}+2p_{n}q_{n}^{2}+q_{n}^{3})}{[n]_{q_{n}}}\right\}\nonumber\\
&=&\lim_{n \to
\infty}\left\{-p_{n}^{n-3}(p^{n}_{n}-q^{n}_{n})(p_{n}-q_{n})+p_{n}^{n-5}(-p_{n}^{3}+3p_{n}q_{n}^{2}+
q_{n}^{3})-\frac{p_{n}^{3n-6}(p_{n}^{2}+p_{n}^{3}+2p_{n}q_{n}^{2}+q_{n}^{3})}{[n]_{q_{n}}}\right\}\nonumber\\
&=&3a.
\end{eqnarray}
Similarly, we can show that;

\begin{eqnarray}\label{ali4}
\lim_{n\to\infty}[n]_{q_{n}}^{2}\{A_{2,n}\}=-6a \,\,\textrm{and}\,\,
\lim_{n\to\infty}[n]_{q_{n}}^{2}\{A_{3,n}\}=3a.
\end{eqnarray}

By combining (\ref{ali1})-(\ref{ali4}), we reach the desired the
result.
\end{proof}

Now, we give a Voronovskaja type theorem for
$B^{(p_{n},q_{n})}_{n,n}(f;x,y)$.

\begin{thm}
Let $f\in C^{2}([0,1]\times[0,1])$. Then, we have
\begin{eqnarray*}
\lim_{n\to\infty}[n]_{p_{n},q_{n}}B^{(p_{n},q_{n})}_{n,n}(f;x,y)-f(x_,y))&=&\frac{(ax-ax^{2})f_{x^{2}}^{\prime\prime}(x,y)}{2}+\frac{(ay-ay^{2})f_{y^{2}}^{\prime\prime}(x,y)}{2}.
\end{eqnarray*}
\end{thm}

\begin{proof}
Let $(x, y)\in [0, 1]\times[0, 1]$. Then, write Taylor's expansion
of $f$ as follows:
\begin{eqnarray}\label{taylor}
f(s,t)&=&f(x,y)+f_{x}^{\prime}(s-x)+f^{\prime}_{y}(t-y)\nonumber\\
&&+\frac{1}{2}\left\{f_{x}^{\prime\prime}(t-x)^{2}+2f_{xy}^{\prime}(s-x)(t-y)+f^{\prime\prime}_{y}(t-y)^{2}\right\}+\varepsilon(s,t)\left((s-x)^{2}+(t-y)^{2}\right)
\end{eqnarray}
where $(s,t)\in [0,1]^{2}$  and $\varepsilon(s,t)\longrightarrow 0$
as $(s,t) \longrightarrow (x,y)$.

Applying  the operator $B^{(p_{n},q_{n})}_{n,n}(f;.)$ on
\eqref{taylor}, we get
\begin{eqnarray*}
B^{(p_{n},q_{n})}_{n,n}(f;s,t)-f(x_,y)&=&f_{x}^{\prime}(x,y)B^{(p_{n},q_{n})}_{n,n}((s-x);x,y)+f_{y}^{\prime}(x,y)B^{(p_{n},q_{n})}_{n,n}((t-y);x,y)\\
&&+\frac{1}{2}\bigg\{f_{x^{2}}^{\prime\prime}B^{(p_{n},q_{n})}_{n,n}((t-x_{0})^{2};x,y)+
2f_{xy}^{\prime}B^{(p_{n},q_{n})}_{n,n}((s-x)(t-y);x,y)\\&&+f^{\prime\prime}_{y^{2}}B^{(p_{n},q_{n})}_{n,n}((t-y)^{2};x,y)\bigg\}
+B^{(p_{n},q_{n})}_{n,n}\left(\varepsilon(s,t)\left((s-x)^{2}+(t-y)^{2}\right);x,y\right).
\end{eqnarray*}
Let us take the limit of both sides of the above equality as
$n\longrightarrow \infty$,
\end{proof}
\begin{eqnarray*}
\lim_{n \to
\infty}[n]_{p_{n},q_{n}}B^{(p_{n},q_{n})}_{n,n}(f;s,t)-f(x,y))&=&\lim_{n
\to\infty}[n]_{p_{n},q_{n}}\frac{1}{2}\bigg\{f_{x^{2}}^{\prime\prime}B^{(p_{n},q_{n})}_{n,n}((t-x)^{2};x,y)\\&&+
2f_{xy}^{\prime}B^{(p_{n},q_{n})}_{n,n}((s-x)(t-y);x,y)+f^{\prime\prime}_{y^{2}}B^{(p_{n},q_{n})}_{n,n}((t-y)^{2};x,y)\bigg\}
\\&&+\lim_{n \to\infty}[n]_{p_{n},q_{n}}B^{(p_{n},q_{n})}_{n,n}\left(\varepsilon(s,t)\left((s-x)^{2}+(t-y)^{2}\right);x,y\right).
\end{eqnarray*}
For the last term on the right hand side, using Cauchy-Schwartz
inequality, we get
\begin{eqnarray*}
B^{(p_{n},q_{n})}_{n,n}\left(\varepsilon(s,t)\left((s-x)^{2}+(t-y)^{2}\right);x,y\right)&\leq&
\sqrt{\lim_{n \to
\infty}B^{(p_{n},q_{n})}_{n,n}\left(\varepsilon^{2}(s,t);x,y\right)}\\
&\times&\sqrt{2\lim_{n
\to\infty}[n]^{2}_{p_{n},q_{n}}B^{(p_{n},q_{n})}_{n,n}\left(\varepsilon(s,t)\left((s-x)^{4}+(t-y)^{4}\right);x,y\right)}.
\end{eqnarray*}
As $\lim_{n \to
\infty}B^{(p_{n},q_{n})}_{n,n}\left(\varepsilon^{2}(s,t);x,y\right)=\varepsilon^{2}(x,y)=0$
and using Lemma \ref{lemmaVAR}(ii)

$\lim_{n\to\infty}[n]^{2}_{p_{n},q_{n}}B^{(p_{n},q_{n})}_{n,n}\left((s-x)^{4}+(t-y)^{4});x,y\right)$
is finite, then we obtain
\begin{eqnarray*}
\lim_{n
\to\infty}[n]^{2}_{p_{n},q_{n}}B^{(p_{n},q_{n})}_{n,n}\left(\varepsilon(s,t)\left((s-x)^{4}+(t-y)^{4}\right);x,y\right)=0.
\end{eqnarray*}
Hence, one can see that

\begin{eqnarray*}
\lim_{n\to\infty}[n]_{p_{n},q_{n}}B^{(p_{n},q_{n})}_{n,n}(f;x,y)-f(x_,y))&=&\frac{(ax-ax^{2})f_{x^{2}}^{\prime\prime}(x,y)}{2}+\frac{(ay-ay^{2})f_{y^{2}}^{\prime\prime}(x,y)}{2}.
\end{eqnarray*}
This step completes the proof.

\end{proof}

\end{document}